\documentclass[11pt]{article}
\title{\textsc{On the size of lattice simplices\\ with a single interior lattice point}}
\author{Gennadiy Averkov\footnote{Institute for Mathematical Optimization, Faculty of Mathematics, University of Magdeburg, Universit\"atsplatz 2,  39106 Magdeburg, Germany. Email: averkov@math.uni-magdeburg.de}}
\usepackage{amsthm,amsmath,amssymb,enumerate,graphicx,graphpap,empheq}

\newcommand{\rmcmd}[1]{\mathop{\mathrm{#1}}\nolimits}
\newcommand{\conv}{\rmcmd{conv}}
\newcommand{\aff}{\rmcmd{aff}}

\newcommand{\real}{\mathbb{R}}
\newcommand{\natur}{\mathbb{N}}
\newcommand{\integer}{\mathbb{Z}}
\newcommand{\Aff}{\rmcmd{Aff}}
\newcommand{\card}[1]{\left|#1\right|}
\newcommand{\term}[1]{\emph{#1}}

\newcommand{\cT}{\mathcal{T}}

\newcommand{\setcond}[2]{\left\{#1 \, : \, #2 \right\}}
\newcommand{\setcondsep}{:}

\newcommand{\intr}{\rmcmd{int}}

\newcommand{\textand}{\text{and}}
\newcommand{\textor}{\text{or}}
\newcommand{\textforall}{\text{for all}}
\newcommand{\bL}{\mathbb{L}}
\newcommand{\bM}{\mathbb{M}}
\newcommand{\bV}{\mathbb{V}}

\newcommand{\vol}{\rmcmd{vol}}

\newcommand{\header}[1]{\textup{(#1)}}
\newcommand{\cL}{\mathcal{L}} 

\newenvironment{figtabular}[1]{
	\begin{figure}[htb]
		\begin{center}
			\begin{tabular}{#1}
}{
			\end{tabular}

		\end{center}

	\end{figure}
}

\newtheorem{nn}{}[section]
\newtheorem{theorem}[nn]{Theorem}

\newtheorem{lemma}[nn]{Lemma}
\newtheorem{remark}[nn]{Remark}

\newtheorem{example}[nn]{Example}

\theoremstyle{definition}
\newtheorem*{acknowledgements*}{Acknowledgements}

\numberwithin{equation}{section}

\begin{document}

\maketitle

\begin{abstract}
	Let $\cT^d$ be the set of all $d$-dimensional simplices $T$  in $\real^d$ with integer vertices and a single integer point in the interior of $T$. It follows from a result of Hensley that $\cT^d$ is finite up to affine transformations that preserve $\integer^d$. It is known that, when $d$ grows, the maximum volume of the simplices $T \in \cT^d$ becomes extremely large. We improve and refine bounds on the size of $T \in \cT^d$ (where by the size we mean the volume or the number of lattice points). It is shown that each $T \in \cT^d$ can be decomposed into an ascending chain of faces $G_1 \subseteq \cdots \subseteq G_d=T$ such that, for every $i \in \{1,\ldots,d\}$, $G_i$ is $i$-dimensional and the size of $G_i$ is bounded from above in terms of $i$ and $d$. The bound on the size of $G_i$ is double exponential in $i$. The presented upper bounds are asymptotically tight on the log-log scale.
\end{abstract}

\newtheoremstyle{itsemicolon}{}{}{\mdseries\rmfamily}{}{\itshape}{:}{ }{}
\newtheoremstyle{itdot}{}{}{\mdseries\rmfamily}{}{\itshape}{:}{ }{}
\theoremstyle{itdot}
\newtheorem*{msc*}{2010 Mathematics Subject Classification} 

\begin{msc*}
	Primary 52B20;  Secondary 14J45, 90C11
\end{msc*}

\newtheorem*{keywords*}{Keywords}

\begin{keywords*}
	barycentric coordinates; inequality; lattice-point enumerator; lattice polytope; Minkowski's fundamental theorem; simplex; volume
\end{keywords*}

\section{Introduction}

For standard background from the \term{geometry of numbers} and \term{convex geometry} we refer to \cite{MR893813,MR1216521,MR1940576,MR2335496}.  The cardinality of a set $X$ is denoted by $\card{X}$.  Throughout the paper we fix a $d$-dimensional vector space $\bV$ over $\real$ with $d \in \natur$ and a lattice $\bL$ of rank $d$ in $\bV$.   In most of the cases it will be convenient to use coordinate-free notation. However, when we use analytic expressions with matrices or the notion of volume, we fix coordinates. In such cases $\bV$ is identified with $\real^d$, the set of columns of $d$ real numbers. If $\bM$ is a lattice of rank $i \in \{1,\ldots,d\}$, then the \term{determinant} of $\bM$ (denoted by $\det \bM$) is the $i$-dimensional volume of any Dirichlet cell of $\bM$. A subset $P$ of $\bV$ is said to be a \term{lattice polytope} (with respect to the lattice $\bL$) if $P$ is the convex hull of finitely many points of $\bL$. Let $P$ be an $i$-dimensional lattice polytope in $\bV$, where $i \in \{1,\ldots,d\}$. Then by $\vol(P)$ we denote the $i$-dimensional \term{volume} of $P$. We also use the $i$-dimensional \term{normalized volume} of $P$ defined by 
	\[ 
		\vol_{\bL}(P):= \frac{\vol(P)}{\det \bM},
	\]	
	where $\bM := \bL \cap X$ and $X$ is the linear hull of vectors $x-y$ with $x,y \in P$. We remark that $\vol_{\bL}(P)$ does not depend on the choice of coordinates in $\bV$.

Let $\cT^d$ denote the family of all $d$-dimensional lattice simplices in $\bV$ with precisely one interior lattice point. It was shown by Hensley \cite{MR688412} that the volume of the elements of $\cT^d$ is bounded by a constant depending only on $d$. Continuing the research initiated in \cite{MR688412}, we present bounds on the size of lattice simplices with precisely one interior lattice point. 

Note that the family $\cT^d$ appears naturally in \term{algebraic geometry}; see \cite{MR1166957,arxiv:math/0001109,MR2549542}.
Furthermore, as seen from \cite{arXiv:1010.1077}, $\cT^d$ can be used in the study of the class $\cL^d$ of inclusion-maximal $d$-dimensional lattice polytopes without interior lattice points. The elements of $\cL^d$ are important for the \term{cutting-plane theory} in integer and mixed integer optimization (see \cite{arXiv:1010.1077,DelPiaWeismantel10} and the references therein).

We prove the following theorem, which is our main tool.

\begin{theorem} \label{bary:ineq:alt}
	Let $T \in \cT^d$. Let $\beta_i$ with $i \in \{0,\ldots,d\}$ be the barycentric coordinates of the unique interior lattice point of $T$. Then, for every partition $(I,J)$ of $\{0,\ldots,d\}$, one has 
	\begin{equation} \label{sum:prod:ineq:1}
		\sum_{i \in I} \beta_i \ge \prod_{j \in J} \beta_j.
	\end{equation}
\end{theorem}

In Theorem~\ref{bary:ineq:alt} without loss of generality one can assume $\beta_0 \ge \cdots \ge \beta_d > 0$. Then, using \eqref{sum:prod:ineq:1}, it is possible to determine lower bounds for $\beta_i$'s in terms of $i$ and $d$. As a consequence we obtain the following result on the maximum size of simplices from $\cT^d$. 

\begin{theorem} \label{balance:of:simplices}
	Let $d \in \natur$. Then the following statements hold.
	\begin{enumerate}[I.]
		\item \label{main:upper:bounds} For every $T \in \cT^d$ there exist faces $G_1 \subseteq \cdots \subseteq G_d = T$ of $T$ such that, for every $i \in \{1,\ldots,d\}$, $G_i$ is $i$-dimensional and satisfies
		\begin{align}
			\vol_{\bL}(G_i) & \le \frac{1}{i!} (d+1)^{2^i-1}, \label{vol:upper} \\
			\card{G_i \cap \bL} & \le i + (d+1)^{2^i-1}. \label{card:upper}
		\end{align}
		\item There exists $T' \in \cT^d$ and faces $G_1 \subseteq \cdots \subseteq G_d = T'$ of $T'$ such that, for every $i \in \{1,\ldots,d\}$, $G_i$ is $i$-dimensional and satisfies
		\begin{align}
			\vol_{\bL}(G_i) & \ge \frac{1}{i!} (2^{2^{i-1}}-1) , \label{vol:lower} \\
			\card{G_i \cap \bL} & \ge 2^{2^{i-2}}. \label{card:lower}
		\end{align}
	\end{enumerate}
\end{theorem}

Loosely speaking, Theorem~\ref{balance:of:simplices} asserts that the largest simplices from $\cT^d$ are extremely anisotropic. Part~I contains the main message of Theorem~\ref{balance:of:simplices}. Part~II is given as a complement estimating the quality of the upper bounds from Part~I. For $d \in\natur$ let
\[
	v(d) := \sup_{T \in \cT^d} \vol_\bL(T)
\]
By Theorem~\ref{balance:of:simplices}, we obtain $ v(d)= (d+1)^{O(2^d)}$. This improves the best available bound $v(d) = (2^d)^{O(2^d)}$, which can  be found in \cite[(10)]{MR1996360}.

Let us give a short overview of results related to Theorem~\ref{balance:of:simplices}. The study of upper bounds on the size of lattice polytopes with a fixed positive number of interior lattice points was initiated in \cite{MR688412} and continued in \cite{MR1138580} and \cite{MR1996360}. The sources \cite{MR1116368} and \cite{MR1166957} present qualitative arguments which show $v(d)<+\infty$ for every $d \in \natur$ (see also \cite[\S4]{arxiv:math/0001109} and \cite[\S7]{MR1996360}). Analoga of the results from \cite{MR688412,MR1138580,MR1996360} for lattice polytopes without interior lattice points were given in \cite{MR2599087,arXiv:1010.1077,arXiv:1101.4292}. The sources \cite{MR1039134} and \cite{MR2760660} provide classification of $d$-dimensional lattice polytopes with precisely one interior lattice point for the cases $d=2$ and $d=3$, respectively. Inequalities involving the area of lattice polygons can be found in \cite{MR0066430,MR0430960,MR2478059}.

\section{Preliminaries}

The abbreviations $\aff, \conv$ and $\intr$ stand for the affine hull, convex hull and interior, respectively. The origin of $\bV$ is denoted by $o$. By $e_1,\ldots,e_d$ we denote the standard basis of $\real^d$. The $l_\infty$-norm is denoted by  $\| \cdot \|_\infty$. We shall use the following three well-known results from the geometry of numbers.

\begin{theorem} \header{Minkowski's fundamental theorem.} \label{mink:1st}
	Let $U$ be an open convex set in $\bV$ with $U=-U$ and $\vol_{\bL}(U) > 2^d$. Then $U$ contains a nonzero element of $\bL$.
\end{theorem}

\begin{theorem} \label{mink:cor}
	Let $A$ be a real $d \times d$ matrix with $0<|\det A| <1$. Then there exists  $x \in \integer^d \setminus \{o\}$ satisfying $\|A x \|_\infty <1$.
\end{theorem}

\begin{theorem} \header{Blichfeldt's theorem.} \label{blichfeldt} Let $K$ be a compact convex set in $\bV$ containing $k$ points of $\bL$ that linearly span $\bV$. Then $d+d! \vol_\bL(K) \ge k$.
\end{theorem}

Note that Theorem~\ref{mink:cor} is a straightforward consequence of Theorem~\ref{mink:1st}.

Let $T$ be a $d$-dimensional simplex in $\bV$ and let $p_0,\ldots,p_d$ be the vertices of $T$. Then every point $x \in \bV$ can be uniquely given by the values $\alpha_0,\ldots,\alpha_d \in \real$ satisfying
\begin{align}
	x & =\alpha_0 p_0 + \cdots + \alpha_d p_d,  \label{x=bary:rep}\\
	1 & = \alpha_0 + \cdots + \alpha_d. \nonumber
\end{align}
The values $\alpha_0,\ldots,\alpha_d$ are said to be the \term{barycentric coordinates} of $x$ with respect to the simplex $T$. The barycentric coordinates have the following geometric interpretation. Let $H_i$ be the affine hull of the facet of $T$ opposite to $p_i$. Then $\alpha_i$ is the \term{signed weighted distance from $x$ to $H_i$}:  $|\alpha_i|$ is the ratio of the distance of $x$ to $H_i$ to the distance of $p_i$ to $H_i$, and $\alpha_i >0$ if and only if $x$ and $p_i$ are in the same open halfspace defined by $H_i$. In particular, we have:
\begin{itemize}
	\item $x$ lies in the interior of $T$ if and only if $\alpha_i>0$ for all $i \in \{0,\ldots,d\}$. 
	\item For $i \in \{0,\ldots,d\}$, $x \in H_i$ if and only if $\alpha_i=0$.
\end{itemize}
The barycentric coordinates $\alpha_0,\ldots,\alpha_d$ of $x$ can be found from
\begin{equation} \label{bary:coord:in:matrix:form}
	\begin{bmatrix}
		p_0 & \cdots & p_d \\
		1 & \cdots & 1
	\end{bmatrix}
	\begin{bmatrix}
		\alpha_0 \\
		\vdots \\
		\alpha_d
	\end{bmatrix}
	= 
	\begin{bmatrix}
		x \\
		1
	\end{bmatrix}.
\end{equation}
(We recall that in matrix expressions we assume $\bV=\real^d$.) Since $p_0,\ldots,p_d$, are affinely independent, the $(d+1) \times (d+1)$-matrix in \eqref{bary:coord:in:matrix:form} is invertible. Thus, we can express $\alpha_i$'s by
\begin{equation*}
	\alpha_i:= e_i^\top 
	{\begin{bmatrix}
		p_0 & \cdots & p_d \\
		1 & \cdots & 1
	\end{bmatrix}}^{-1}
	\begin{bmatrix}
		x \\
		1
	\end{bmatrix}.
\end{equation*}
It follows that for every $i \in \{0,\ldots,d\}$ we have $\alpha_i:=l_i(x)$, where $l_i(x)$ is an affine-linear function in $x$. (A real-valued function $f$ on $\bV$ is called \term{affine-linear} if $f((1-\lambda) x+ \lambda y)=(1-\lambda) f(x) + \lambda f(y)$ for all $x, y \in \bV,\lambda \in \real$.) The functions $l_0,\ldots,l_d$ are uniquely determined by  
\begin{align} \label{l:i:def}
	l_i(p_j) & = \delta_{i,j} & & \forall \ i,j \in \{0,\ldots,d\},
\end{align}
where $\delta_{i,j}$ denotes the \term{Kronecker delta}. This follows by applying $l_i$ to \eqref{x=bary:rep} and then using the affine-linearity.

We shall use pairs $(I,J)$ of disjoint subsets of $\{0,\ldots,d\}$ with $I \cup J = \{0,\ldots,d\}$. Such a pair $(I,J)$ is called a \term{partition} if both $I$ and $J$ are nonempty. Every face of $T$ can be represented by
\begin{align} \label{F:I:def}
	F_I:= & \setcond{x \in T}{l_i(x)=0 \ \  \forall \ i \in I}  \\
		= & \conv \setcond{p_j}{j \in J} \nonumber
\end{align}
 for an appropriate pair $(I,J)$ as above. Clearly, partitions $(I,J)$ correspond to proper faces of $T$.

\section{The proofs}

\subsection{Polynomial inequalities for $\beta_i$'s (proof of Theorem~\ref{bary:ineq:alt})}

Let $T \in \cT^d$ and let $\beta_0,\ldots,\beta_d$ be the barycentric coordinates of the  interior lattice point of $T$. Our aim is to provide positive lower bounds for $\beta_i$'s in terms of $d$. Let us sketch how such bounds can be obtained (the idea is due to Hensley \cite[the proof of Lemma~2.3]{MR688412}). Let $p$ be the interior lattice point of $T$. One can show by contradiction that $p$ cannot be too close to the boundary of $T$ (that is, the minimal $\beta_i$ cannot be too small). If $p$ is close to a face $F$ of $T$, then, with the help of Minkowski's first fundamental theorem, one can determine a point $r$ in $\aff F$ whose barycentric coordinates are rational numbers of the form $\frac{m_0}{m},\ldots,\frac{m_d}{m}$ with $m_0,\ldots,m_d \in \integer$ and $m \in \natur$ and such that $r$ is close to $p$. It then turns out that $q:=(m+1) p - m r$ is another integral point in the interior of $T$, yielding a contradiction to $T \in \cT^d$. See also Fig.~\ref{jumping-idea} for an illustration.

\begin{figure}
\unitlength=1.3mm
\begin{center}
%\fbox{
\begin{picture}(30,30)
	\put(0,0){\includegraphics[width=30\unitlength]{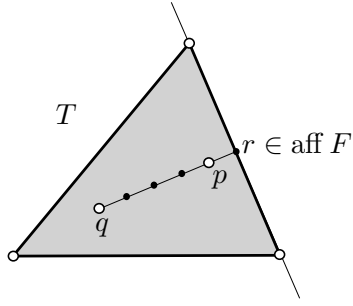}}
	\put(9,7){$q$}
	\put(21,12){$p$}
	\put(24,15){$r \in \aff F$}
	\put(5,18){$T$}
	%\put(0,0){\graphpaper[2](0,0)(30,30)}
\end{picture}
%}
\parbox[h]{0.8\textwidth}{\caption{\label{jumping-idea} Using a lattice point $p \in \intr T$ for construction of another lattice point $q \in \intr T$ in the case that $p$ is close to a face $F$ of $T$. The argument uses a rational point $r \in \aff F$ which approximates $p$ sufficiently well. The point $q$ is obtained by adding $m(p-r)$ to $p$, where $m$ is a common denominator of the barycentric coordinates of $r$ (in the figure one has $m=4$)}}
\end{center}
\end{figure}

\begin{proof}[Proof of Theorem~\ref{bary:ineq:alt}]
	Let $p$ be the interior lattice point of $T$.
	The point $p$ can be given by $p=\sum_{i=0}^d \beta_i p_i$, where $p_0,\ldots,p_d$ are appropriately indexed vertices of $T$. Without loss of generality let $I=\{t,\ldots,d\}$ with $1 \le t \le d$. Consider real unknowns $m_0,\ldots,m_{t-1}$ and $m$ and the vector 
	\[
		\overline{m}:=(m_0,\ldots,m_{t-1},m)^\top
	\]
	If 
	\begin{equation} \label{m:assumptions}
		\overline{m} \in \integer^{t+1}, \ m>0 \  \text{and} \ m=\sum_{i=0}^{t-1} m_i,
	\end{equation}
	then 
	\[
		r=\sum_{i=0}^{t-1} \frac{m_i}{m} p_i
	\]
	is an affine combination (with rational coefficients) of $p_0,\ldots,p_{t-1}$  and 
	\begin{align} 
		q  := & (m +1) p - m r \label{q:def} \\
		   = & (m +1) p - \sum_{i =0}^{t-1} m_i p_i  \nonumber
	\end{align}
	is an affine combination of $p$ and $r$ belonging to $\bL$.
	Below we determine a sufficient condition for the existence of values $m_0,\ldots,m_{t-1},m$ as above for which the point $q$ does not coincide with $p$ and lies in the interior of $T$. If \eqref{m:assumptions} holds, then $(m+1) \beta_i - m_i $ with $i\in \{0,\ldots,t-1\}$ and $(m+1) \beta_i$ with $i \in \{t,\ldots,d\}$ are the barycentric coordinates of $q$. Thus (taking into account the geometric meaning of the barycentric coordinates) we obtain that, assuming \eqref{m:assumptions}, the following conditions are equivalent:
	\begin{enumerate}[(i)]
		\item $q\in \intr T$;
		\item $(m  +1) \beta_i - m_i  > 0$ for every $i \in \{0,\ldots,t-1\}$. 
	\end{enumerate}
	
	Using (i) $\Leftrightarrow$ (ii) we can formulate the following system of conditions on $\overline{m}$ sufficient for $q$ to be a lattice point in $\intr T$ with $q \ne p$:

		\begin{empheq}[left=\empheqlbrace]{align}
			 \overline{m} & \in \integer^{t+1}, & \label{vec:k:int} \\
			 m & > 0, & \label{k>0} \\
			 \frac{1}{\beta_i} m_i - m & < 1 & \forall \ i \in \{0,\ldots,t-1\}, \label{k:single:k_i:nonsym}\\
			 m - \sum_{i = 0}^{t-1} m_i & = 0. \label{k=:eq}
		\end{empheq}

	We also consider another system somewhat similar to \eqref{vec:k:int}--\eqref{k=:eq}:

	\begin{empheq}[left=\empheqlbrace]{align}
			\overline{m} & \in \integer^{t+1}, & \label{vec:k:int:again} \\
			\overline{m} & \ne o, & \label{k_i:k:nontrivial}\\
			 \biggl|\frac{1}{\beta_i} m_i - m \biggl| &<1 & \forall \ i \in \{0,\ldots,t-1\}, \label{k:single:k_i:sym} \\
			 \biggl|m - \sum_{i=0}^{t-1} m_i\biggr| &< 1. & \label{k almost =:eq}
	\end{empheq}

	For \eqref{vec:k:int:again}--\eqref{k almost =:eq} Theorem~\ref{mink:cor} can be applied, but first we investigate the relation of \eqref{vec:k:int:again}--\eqref{k almost =:eq} to \eqref{vec:k:int}--\eqref{k=:eq}. We claim that  if  \eqref{vec:k:int:again}--\eqref{k almost =:eq} is solvable with respect to $\overline{m}$, then also \eqref{vec:k:int}--\eqref{k=:eq} is solvable with respect to $\overline{m}$. Assume that conditions \eqref{vec:k:int:again}--\eqref{k almost =:eq} are fulfilled. Then $m$ is not zero. In fact, if $m=0$, then \eqref{k:single:k_i:sym} together with \eqref{vec:k:int:again} and $0<\beta_i < 1$ implies $m_i=0$ for every $i \in \{0,\ldots,t-1\}$. This yields a contradiction to \eqref{k_i:k:nontrivial}. Hence $m \ne 0$. The system \eqref{vec:k:int:again}--\eqref{k almost =:eq} is invariant under the change of $\overline{m}$ to $-\overline{m}$. Thus, we can assume $m>0$. Condition \eqref{k almost =:eq} together with  \eqref{vec:k:int:again} implies \eqref{k=:eq}. The implication from \eqref{k:single:k_i:sym} to \eqref{k:single:k_i:nonsym} is trivial. This verifies the claim.

	Let us reformulate the system \eqref{vec:k:int:again}--\eqref{k almost =:eq} in matrix terms. We introduce the matrix
	
	\begin{align*}
		B  & := 
		\begin{bmatrix}
			\frac{1}{\beta_0} &  &  & -1 \\
			& \ddots & & \vdots \\
			& &  \frac{1}{\beta_{t-1}} & -1 \\
			-1 & \cdots & -1 & 1 
		\end{bmatrix}	
	\end{align*}
	of size $(t+1) \times (t+1)$ (the submatrix of $B$ generated by the first $t$ rows and $t$ columns is diagonal, with diagonal entries $\frac{1}{\beta_0},\ldots,\frac{1}{\beta_{t-1}}$). The system \eqref{vec:k:int:again}--\eqref{k almost =:eq} can thus be formulated as
	\begin{empheq}[left=\empheqlbrace]{align}
			 \overline{m} & \in \integer^{t+1}, \label{vec:k:int:once:again} \\ 
			 \overline{m} & \ne o, \\
			 \| B \overline{m} \|_\infty & < 1. \label{inf:norm:cond}
	\end{empheq}

	We compute the determinant of $B$ by transforming the matrix into an upper triangular form: all $-1$'s in the last row can be turned to $0$ by adding an appropriate linear combination of the first $t$ rows of $B$. Thus
	\begin{equation} \label{ratio}
		\det B =  \frac{1-\sum_{i =0}^{t-1} \beta_i }{\prod_{j = 0}^{t-1} \beta_j} =  \frac{\sum_{i \in I}\beta_i}{\prod_{i \in J} \beta_j}.
	\end{equation}
	Hence $\det B \ge 0$ and since $I \ne \emptyset$, we even have $\det B > 0$. We show $\det B \ge 1$ by contradiction. Assume that $\det B<1$. Then, by Theorem~\ref{mink:cor}, \eqref{vec:k:int:once:again}--\eqref{inf:norm:cond} is solvable. The system \eqref{vec:k:int:once:again}--\eqref{inf:norm:cond} is a matrix form of \eqref{vec:k:int:again}--\eqref{k almost =:eq}. Above we showed that solvability of \eqref{vec:k:int:again}--\eqref{k almost =:eq} implies solvability of \eqref{vec:k:int}--\eqref{k=:eq}. Thus, there exists $\overline{m}$ satisfying \eqref{vec:k:int}--\eqref{k=:eq}. With such $\overline{m}$ the lattice point $q$ given by \eqref{q:def} does not coincide with $p$ and lies in the interior of $T$. This yields a contradiction to $T \in \cT^d$. Thus, $\det B \ge 1$ and the assertion follows from \eqref{ratio}.
\end{proof}

Theorem~\ref{bary:ineq:alt} presents exponentially many inequalities for the barycentric coordinates of the interior lattice point: one for each $I$ with $\emptyset \varsubsetneq I \varsubsetneq \{0,\ldots,d\}$. It turns out that most of these inequalities are redundant. In fact, assume that for $i \in I$ and $j \in J$ one has $\beta_{i} > \beta_{j}$. Then we can modify $(I,J)$ by moving $i$ from $I$ to $J$ and moving $j$ from $J$ to $i$. By this change we lower the left hand side and raise the right hand side of \eqref{sum:prod:ineq:1} making the inequality tighter. We thus see that the inequalities for the partitions $(I,J)$ satisfying $\beta_{i} \le \beta_{j}$ for every $i \in I$ and every $j \in J$ imply the inequalities for all the remaining partitions. Such strongest inequalities can be described as follows. After an appropriate reindexing, we can assume that the barycentric coordinates $\beta_0,\ldots,\beta_d$ are in an ordered sequence, say $\beta_0 \ge \cdots \ge \beta_d > 0$. Then, taking into account the redundancy, only the following $d$ inequalities remain:

\begin{equation} \label{sum:prod:ineq}
	\sum_{i=j+1}^d \beta_i \ge \prod_{i=0}^j \beta_i
\end{equation}
where $j \in \{0,\ldots,d-1\}$.

\subsection{Lower bounds for $\beta_i$'s}

Below we use \eqref{sum:prod:ineq} to derive lower bounds for the barycentric coordinates $\beta_0,\ldots,\beta_d$. First we observe that the inequalities $\beta_0 \ge \cdots \ge \beta_d$ together with $\beta_0+ \cdots + \beta_d=1$ yield $\beta_0 \ge \frac{1}{d+1}$. Inequality \eqref{sum:prod:ineq} can be relaxed to
\begin{equation} \label{suc:pred:ineq}
	(d+1) \beta_{j+1} \ge \prod_{i=0}^j \beta_i
\end{equation}
On the left hand side of \eqref{suc:pred:ineq} we could use a smaller factor instead of $d+1$, but with the factor $d+1$ the lower bounds for $\beta_j$'s that we give below can be expressed by simpler formulas. By \eqref{suc:pred:ineq} each $\beta_{j}$ is bounded in terms of its `predecessors' $\beta_0,\ldots,\beta_{j-1}$. Thus, consecutively applying \eqref{suc:pred:ineq} we can bound every $\beta_j$ in terms of $j$ and $\beta_0$. Taking into account  $\beta_0 \ge \frac{1}{d+1}$, we then see that each $\beta_j$ can be bounded from below in terms of $j$ and $d$. In this way we arrive at Part~I of the following

\begin{theorem} \label{beta:bounds}
	Let $d \in \natur$. Then the following holds:
	\begin{enumerate}[I.]
 		\item For every $T \in \cT^d$ the (ordered) barycentric coordinates $\beta_0 \ge \cdots \ge \beta_d > 0$ of the unique interior lattice point of $T$ satisfy
		\begin{equation} \label{lower:bound:for:beta}
			\beta_i \ge (d+1)^{-2^i}
		\end{equation}
		for every $i \in \{0,\ldots,d\}$.
		\item \label{beta:for:zpw} There exists $T' \in \cT^d$ such that the barycentric coordinates $\beta_0 \ge \cdots \ge \beta_d > 0$ of the unique interior lattice point of $T'$ satisfy 
		\[
			\beta_i \le \frac{1}{2^{2^{i-1}}-1}.
		\]
		for every $i \in \{0,\ldots,d\}$.
	\end{enumerate}
\end{theorem}

Part~II of Theorem~\ref{beta:bounds} is not used in the proof of the main result. By Part~II we only wish to estimate the quality of the upper bounds in Part~I.

\begin{proof}[Proof of Part~I of Theorem~\ref{beta:bounds}] We use induction on $i$. For $i=0$, \eqref{lower:bound:for:beta} is fulfilled. Assume the inequalities hold for all $i \in \{0,\ldots,j\}$ with some $j \in \{0,\ldots,d-1\}$. Using \eqref{suc:pred:ineq} and the inductive assumption we obtain 
\begin{align*}
	\beta_{j+1} \ge (d+1)^{-1} \prod_{i=0}^j (d+1)^{-2^i} = (d+1)^{-(1+\sum_{i=0}^j 2^i)} = (d+1)^{-2^{j+1}}.
\end{align*}
Thus, \eqref{lower:bound:for:beta} has been verified for $i=j+1$.
\end{proof}

For showing Part~II of Theorem~\ref{beta:bounds} we use the following

\begin{example} \header{A large simplex in $\cT^d$.} \label{zpw-example}
Let $\bV=\real^d$ and $\bL=\integer^d$. Below we introduce the simplex $T' \in \cT^d$ given in \cite[p.~189]{MR688412}. The simplex $T'$ is a slight modification of the simplex constructed by Zaks, Perles and Wills \cite{MR651251}. We first introduce a recursive sequence $(t_n)_{n \in \natur}$ by

\begin{equation*}
	t_n := 
	\begin{cases} 
		2 & \text{for} \ n=1, \\
		t_{n-1}^2 - t_{n-1} + 1 & \text{for} \ n \ge 2.
	\end{cases}
\end{equation*}

The sequence can also be defined by 

\begin{equation} \label{t:n:another:def}
	t_n = 
	\begin{cases} 
		2 & \text{for} \ n=1, \\
		1+ \prod_{i=1}^{n-1} t_i  & \text{for} \ n \ge 2.
	\end{cases}
\end{equation}

The following relations for $t_n$ will be useful:
\begin{align} 
	& 2^{2^{n-2}} \le t_n \le 2^{2^{n-1}}, \label{t:n:bounds} \\
	& \frac{1}{t_1}+ \ldots + \frac{1}{t_n}+\frac{1}{t_{n+1}-1}=1. \label{t:n:equation}
\end{align}
Relations \eqref{t:n:another:def}, \eqref{t:n:bounds} and \eqref{t:n:equation} can be found in \cite[p.~1026]{MR1138580}. We define the simplex 
\begin{equation*}
	T' := \conv \{o,t_1 e_1,\ldots, t_d e_d \}.
\end{equation*}
Hensley \cite[p.~189]{MR688412} noticed that $e_1+ \cdots + e_d$ is the unique interior lattice point of $T'$. Thus, $T' \in \cT^d$. The bounds \eqref{t:n:bounds} show that $T'$ is large for large $d$.
\end{example}

\begin{proof}[Proof of Part~\ref{beta:for:zpw} of Theorem~\ref{beta:bounds}]
	Consider $T'$ as in Example~\ref{zpw-example}. From \eqref{t:n:equation} we see that $\frac{1}{t_1},\ldots,\frac{1}{t_d}$ and $\frac{1}{t_{d+1}-1}$ are the barycentric coordinates of the interior lattice point of $T'$. The assertion follows from the lower bounds for $t_n$ given in \eqref{t:n:bounds}.
\end{proof}

\begin{remark}
	For $T \in \cT^d$ let $\beta_d(T)$ be the smallest barycentric coordinate of the interior lattice point of $T$. By Theorem~\ref{beta:bounds}.I. we show that
	\begin{equation} \label{c(d)>0}
		c(d):= \inf_{T \in \cT^d} \beta_d(T) \ge (d+1)^{-2^d}.
	\end{equation} 
	From the arguments given in the remainder of the paper it will be seen that the finiteness of $\cT^d / \Aff(\integer^d)$ is equivalent to the inequality $c(d)>0$. Pikhurko \cite[Section~7]{MR1996360} indicates that $c(d)>0$ follows from the arguments of Lawrence \cite[Lemma~5]{MR1116368}.  However, \cite{MR1116368} does not yield any explicit positive lower bound for $c(d)$ for a general $d$.
\end{remark}

\begin{remark} \header{Quality of the lower bounds for $\beta_i$'s.}
	The bound \eqref{lower:bound:for:beta} is tight for $i=0$. This can be seen by considering the simplices $\conv \{o,(d+1) e_1,\ldots,(d+1) e_d\}$ and $\conv \{-(e_1+ \cdots + e_d), e_1,\ldots,e_d\}$ (with respect to $\bL=\integer^d$) whose unique interior lattice point is the center of mass. For large $d$'s  \cite[Lemma~2.2]{MR1138580} provides the better bound $\beta_d \ge 14^{-2^{d+1}}$ than our bound $\beta_d \ge (d+1)^{-2^d}$ contained in \eqref{lower:bound:for:beta}. From this we see that the bound \eqref{lower:bound:for:beta} seems to become less tight as $i$ moves from $0$ to $d$. Still it should be mentioned that our lower bound for $\beta_d$ is better than the bound from \cite[Lemma~2.2]{MR1138580} for $1 \le d < 14^2-1 = 195$.

	The question on the determination of $c(d)$ given by \eqref{c(d)>0} was raised by Hensley \cite[p.~189]{MR688412}, who conjectured that the infimum in \eqref{c(d)>0} is attained for the simplex $T'$ from Example~\ref{zpw-example}. This conjecture is true for simplices $T \in \cT^d$ of the form $T=\conv \{o,s_1 e_1,\ldots,s_d e_d\}$ with $s_1,\ldots,s_d \in \natur$, as follows from \cite{MR1520110} and \cite{MR0043117}.
\end{remark}

\begin{figtabular}{c}
 
\unitlength=1.7mm
%\fbox{
\begin{picture}(30,23)
	\put(0,-5){\includegraphics[width=30\unitlength]{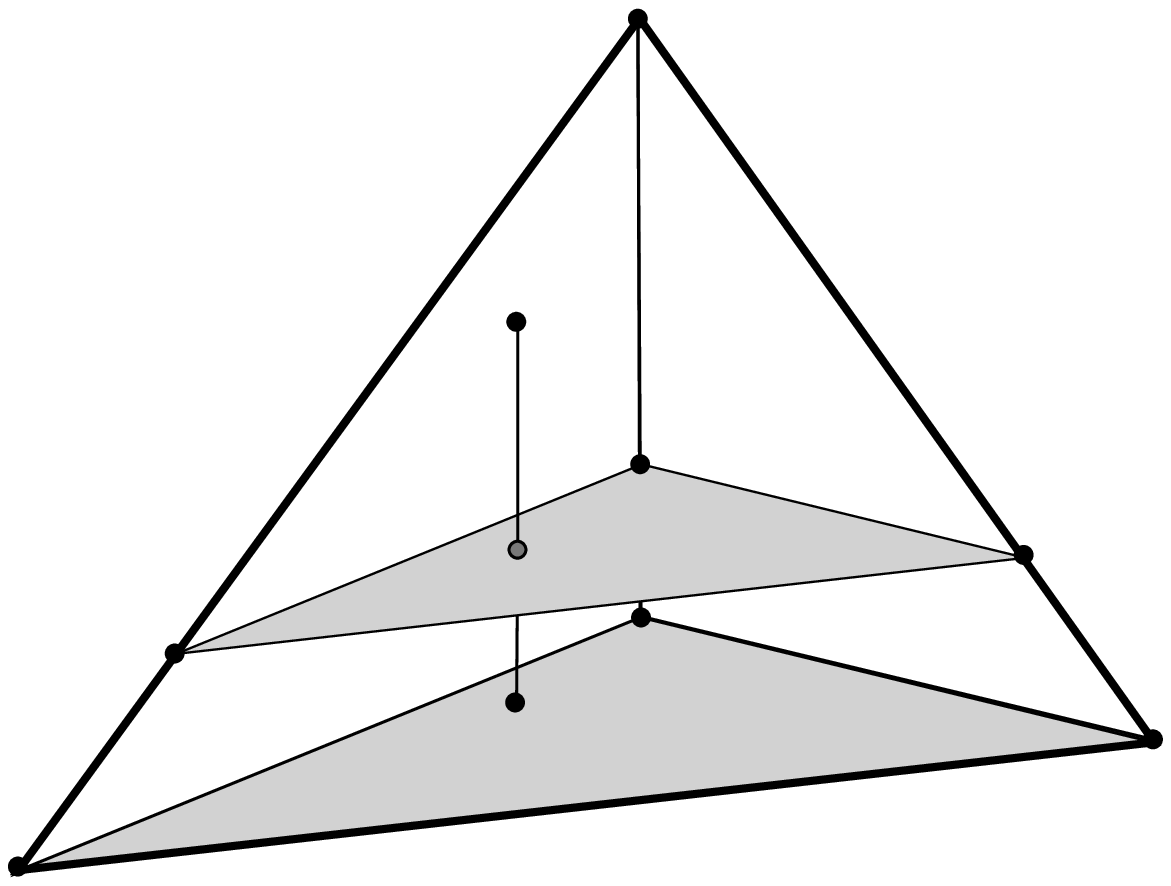}}
	\put(6,16){$T$}
	\put(14,8){$p$}
	%\put(0,0){\graphpaper[2](0,0)(30,23)}
\end{picture}
%}
\\
\parbox[t]{0.8\textwidth}{\caption{\label{simplex-sections} An interior point $p$ of a simplex $T$, a two-dimensional section of $T$ passing through $p$ and parallel to a facet of $T$ and a one-dimensional section of $T$ passing through $p$ and parallel to an edge of $T$. Lemma~\ref{volumes:of:sections}.III shows how the volumes of such sections can be expressed in terms of the volumes of the faces of $T$ and the barycentric coordinates of $p$.}}

\end{figtabular}

\subsection{Face volumes versus $\beta_i$'s}

 In the following lemma we discuss affine sections of $T$ passing through a fixed interior point of $T$ and parallel to faces of $T$, for an illustration see  Fig.~\ref{simplex-sections}. We shall use the affine-linear functions defined by \eqref{l:i:def}.

\begin{lemma} \label{volumes:of:sections}
	Let $T$ be a $d$-dimensional simplex in $\bV$ with vertices $p_0,\ldots,p_d$. Let $p$ be a point in $\intr T$ with barycentric coordinates $\beta_0,\ldots,\beta_d$. Let $I,J$ be disjoint sets with $I \cup J = \{0,\ldots,d\}$ and $J \ne \emptyset$ and let
	\begin{align}
		T_I &:=\setcond{x \in T}{l_i(x) = \beta_i \ \ \forall \ i \in I}, \label{T_I:def} \\
		\beta_J &:= \sum_{j \in J} \beta_j. \label{beta_J:def}
	\end{align}
	Then the following holds: 
	\begin{enumerate}[I.]
		\item The set $T_I$ is a simplex of dimension $\card{J}-1$, and the points $\sum_{i \in I} \beta_i p_i + \beta_{J} p_j$ with $j \in J$ are the vertices of $T_I$.
		\item If $x \in T_I$, then the values $\frac{l_j(x)}{\beta_{J}}$ with $j \in J$ are the barycentric coordinates of $x$ with respect to $T_I$.
		\item One has $\vol(T_I) = \beta_{J}^{\card{J}-1} \vol(F_I)$, where $F_I$ is defined by \eqref{F:I:def}.
	\end{enumerate}
\end{lemma}
\begin{proof} 
Let $x \in T_I$. The values $\frac{l_j(x)}{\beta_{J}}$ with $j \in J$  satisfy
\begin{align*}
	\sum_{j \in J} \frac{l_j(x)}{\beta_{J}} = \frac{1}{\beta_{J}} \left(1-\sum_{i \in I} \beta_i \right) = 1.
\end{align*}

Furthermore
\begin{align*}
	x = \sum_{i=0}^d l_i(x) p_i & = \sum_{i \in I} \beta_i p_i + \sum_{j \in J} l_j(x) p_j \\
	& = \sum_{j \in J} \frac{l_j(x)}{\beta_{J}} \left( \sum_{i \in I} \beta_i p_i+ \beta_{J} p_j \right).
\end{align*}

Thus $x$ is a convex combination of the points $\sum_{i \in I} \beta_i p_i + \beta_{J} p_j$ with $j \in J$. These $\card{J}$ points are affinely independent since they are obtained by applying the nonsingular affine transformation $y \mapsto \sum_{i \in I} \beta_i p_i + \beta_J y$ ($y \in \bV$) to the affinely independent points $p_j$ with $j \in J$. This yields Parts~I and II. Part~III follows in view of the equality $F_I= \conv \setcond{p_j}{j \in J}$.
\end{proof}

The following assertion appears in \cite[Lemma~5]{MR1996360}.

\begin{lemma} \label{vol:bary:lem}
	Let $T$ be a $d$-dimensional simplex in $\bV$ such that the interior of $T$ contains precisely one point of $\bL$. Let $\beta_0,\ldots,\beta_d>0$ be the barycentric coordinates of the unique interior lattice point of $T$ and let $N$ be a subset of $\{0,\ldots,d\}$ of cardinality $d$. Then 
	\[
		\vol_\bL(T) \le \frac{1}{d! \prod_{n \in N} \beta_n}.
	\]
\end{lemma}
\begin{proof}
	Let $p$ be the lattice point in the interior of $T$. We have $p = \sum_{i=0}^d \beta_i p_i$, where $p_i$'s are appropriately indexed vertices of $T$. Without loss of generality assume $N=\{1,\ldots,d\}$. Consider affine-linear functions $l_0,\ldots,l_d$ defined by \eqref{l:i:def}. We have $l_i(p)=\beta_i$ for $i \in \{0,\ldots,d\}$. The simplex $\Tilde{T}:=2p-T$ is a reflection of $T$ in $p$; $\Tilde{T}$ can be described as 
	\[
		\Tilde{T}:=\setcond{x \in \real^d}{\Tilde{l}_0(x) \ge 0, \ldots, \Tilde{l}_d(x) \ge 0}
	\]
	where $\Tilde{l}_i$'s are defined by $\Tilde{l}_i(x):=l_i(2p-x)$ for $x \in \real^d$ and $i \in \{0,\ldots,d\}$. By affine-linearity of $l_0$, for every $x\in \bV$ we have
	\[
		0<\beta_0=l_0(p)=\frac{1}{2}(l_0(x)+l_0(2p-x)) = \frac{1}{2} (l_0(x) + \Tilde{l}_0(x)).
	\]
	Consequently, for every $x \in \bV$, $l_d(x) \ge 0$ or $\Tilde{l}_d(x) \ge 0$. The set 
	\[
		P:=\setcond{x \in \real^d}{l_i(x) \ge 0 \ \textand \ \Tilde{l}_i(x) \ge 0 \ \textforall \ i=1,\ldots,d}
	\]
	is a parallelotope with center at $p$. Let us show $P \subseteq T \cup \Tilde{T}$ (see also Fig.~\ref{box-from-simplex}):

\begin{figure}
\unitlength=1.5mm
\begin{center}
%\fbox{
\begin{picture}(30,25)
	\put(0,0){\includegraphics[width=30\unitlength]{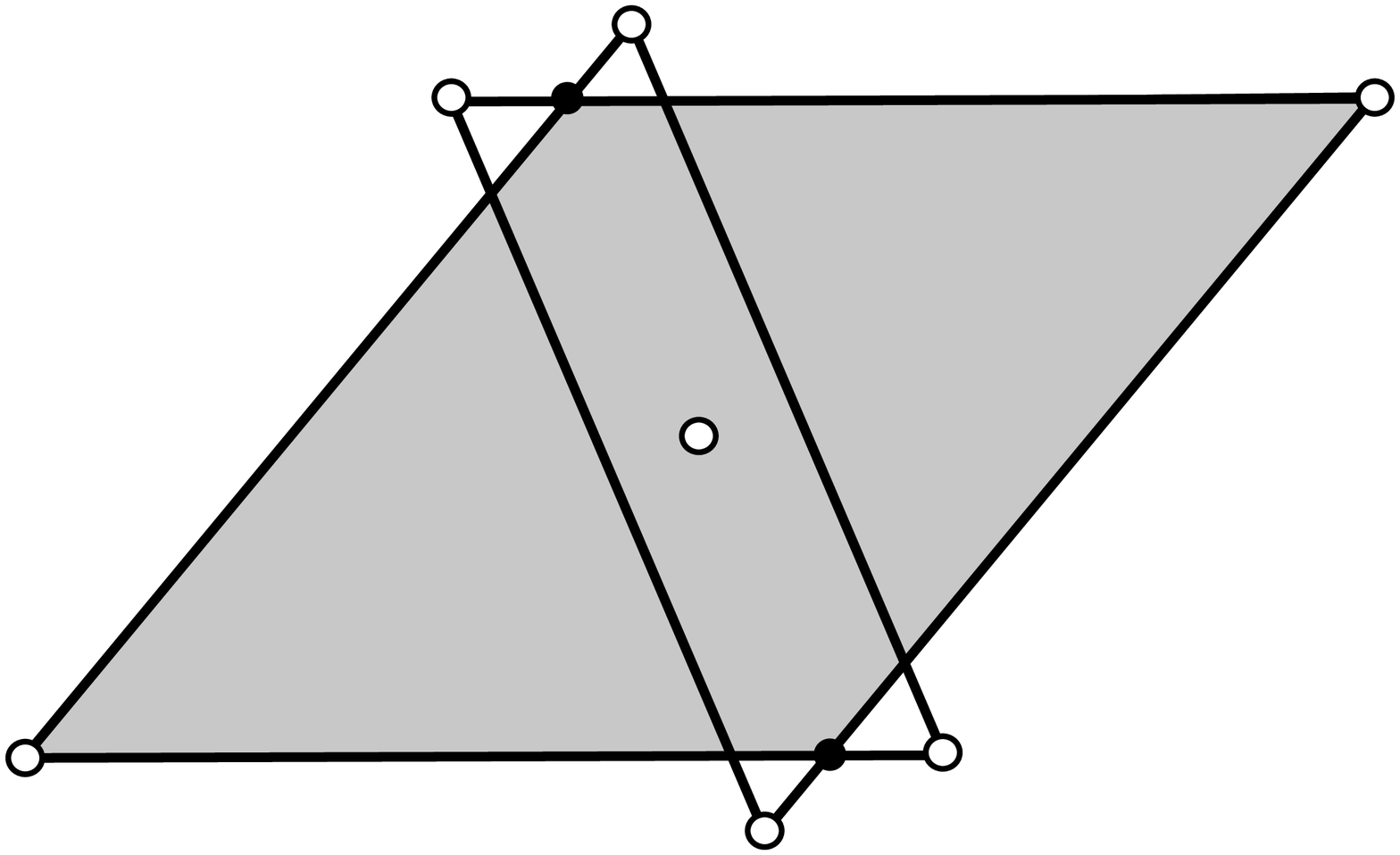}}
	\put(15,7){$p$}
	\put(3,8){$T$}
	\put(20,17){$2p-T$}
	%\put(0,0){\graphpaper[2](0,0)(30,25)}
\end{picture}
\parbox[h]{0.8\textwidth}{\caption{\label{box-from-simplex} Illustration to the proof of Lemma~\ref{vol:bary:lem}. The center of the symmetry $p$ of the parallelotope $P$ (shaded) is the unique interior lattice point of $P$}}
\end{center}
\end{figure}

	\begin{align*}
		P  :=&\setcond{x \in \real^d}{l_1(x) \ge 0, \ldots, l_d(x) \ge 0, \ \Tilde{l}_1(x) \ge 0, \ldots, \Tilde{l}_d(x) \ge 0} \\
			 = & \Bigl\{ x \in \real^d \setcondsep l_1(x) \ge 0, \ldots, l_d(x) \ge 0, \ \Tilde{l}_1(x) \ge 0, \ldots, \Tilde{l}_d(x) \ge 0 \bigr. \\
			& \bigl. \textand \ (l_0(x) \ge 0 \ \textor \ \Tilde{l}_0(x) \ge 0) \Bigr\} \\
			= & \setcond{x \in \real^d}{l_0(x) \ge 0, \ldots, l_d(x) \ge 0, \ \Tilde{l}_1(x) \ge 0, \ldots, \Tilde{l}_d(x) \ge 0} \\
			& \cup \setcond{x \in \real^d}{l_1(x) \ge 0, \ldots, l_d(x) \ge 0, \ \Tilde{l}_0(x) \ge 0, \ldots, \Tilde{l}_{d}(x) \ge 0} \\
			& \subseteq T \cup \Tilde{T}.
	\end{align*}
	Since $p$ is the only interior lattice point of both $T$ and $\Tilde{T}$, we deduce that $p$ is the only interior lattice point of $P$. Due to `affine invariance' of the ratio $\vol(T)/\vol(P)$ for computing $\vol(T)/\vol(P)$ we can assume that $\bV=\real^d$ and $p_0=o, p_1=e_1,\ldots,p_d=e_d$. With this assumption we get $p=(\beta_1,\ldots,\beta_d)^\top$ and $P$ is precisely the set of vectors $x=(x_1,\ldots,x_d)^\top$ for which $0 \le x_i \le 2 \beta_{i}$. Then $\vol(P) = 2^d \prod_{i=1}^{d} \beta_i$, $\vol(T) = \frac{1}{d!}$ and we obtain $\vol(T)/\vol(P) = \frac{1}{d! 2^d \prod_{i=1}^{d} \beta_i}$. By Minkowski's fundamental theorem, $\vol_\bL(P) \le 2^d$. Hence 
	\begin{align*}
		\vol_\bL(T) & = \frac{\vol(T)}{\det \bL} = \frac{\vol(T)}{\vol(P)} \cdot \vol_{\bL}(P) \\ & \le \frac{\vol(T)}{\vol(P)} 2^d = \frac{1}{d! 2^d \prod_{i=1}^{d} \beta_i} 2^d
		 = \frac{1}{d! \prod_{i=1}^{d} \beta_i}.
	\end{align*}
\end{proof}

The following theorem follows directly from Lemma~\ref{vol:bary:lem}.

\begin{theorem} \label{face:vol:versus:beta}
	Let $T \in \cT^d$ and $\beta_0,\ldots,\beta_d$ be the barycentric coordinates of the interior lattice point of $T$. Let $I$ and $N$ be disjoint sets such that $I \cup N$ is a $d$-element subset of $\{0,\ldots,d\}$. Then 
	\[
		\vol_{\bL}(F_I) \le \frac{1}{\card{N}! \prod_{n \in N} \beta_n}.
	\]
\end{theorem}
\begin{proof} Let $J:=\{0,\ldots,d\} \setminus I$. By Lemma~\ref{volumes:of:sections}.III we have $\vol_{\bL}(F_I) = \beta_{J}^{1-\card{J}} \vol_{\bL}(T_I)$ with $T_I$ and $\beta_{J}$ defined by \eqref{T_I:def} and \eqref{beta_J:def}, respectively. By Lemma~\ref{volumes:of:sections}.II, the interior lattice point of $T$ has barycentric coordinates $\frac{\beta_j}{\beta_{J}}$ with $j \in J$ with respect to the simplex $T_I$. Thus, applying Lemma~\ref{vol:bary:lem} (for dimension $\card{J}-1$) to the $(\card{J}-1)$-dimensional simplex $T_I$ we get
\begin{equation*}
	\vol_{\bL}(F_I) = \beta_{J}^{1-\card{J}} \vol_{\bL}(T_I) \le \beta_{J}^{1-\card{J}} \frac{1}{\card{N}! \prod_{n \in N} \frac{\beta_n}{\beta_{J}}}  = \frac{1}{\card{N}! \prod_{n \in N} \beta_n}.
\end{equation*}
\end{proof}

\subsection{The conclusion}

\begin{proof}[Proof of Theorem~\ref{balance:of:simplices}]
	First we show Part~I. Let $\beta_0 \ge \cdots \ge  \beta_d > 0$ be the barycentric coordinates of the unique interior lattice point of $T$. For $i \in \{1,\ldots,d\}$ we define $G_i$ by $G_i:=F_I$, where  $I:=\{0,\ldots,d\} \setminus \{0,\ldots,i\}$. By Theorems~\ref{beta:bounds} and \ref{face:vol:versus:beta} we obtain
	\begin{equation*}
		\vol_\bL(G_i) = \vol_\bL(F_I) \le \frac{1}{i! \prod_{n=0}^{i-1} \beta_n} \le \frac{1}{i!} \prod_{n=0}^{i-1} (d+1)^{2^n} = \frac{1}{i!} (d+1)^{2^i-1}.
	\end{equation*}
	This yields \eqref{vol:upper}. Inequality \eqref{card:upper} follows directly from \eqref{vol:upper} and Blichfeldt's theorem.

	For showing Part~II we consider $\bL=\integer^d$ and the simplex $T'$ from Example~\ref{zpw-example}. For $i \in \{1,\ldots,d\}$ let $G_i := \conv \{o,t_1 e_1,\ldots,t_i e_i\}$. Applying \eqref{t:n:another:def} and \eqref{t:n:bounds} we obtain

	\begin{align*}
		\vol_\bL(G_i) & = \frac{1}{i!} \prod_{j=1}^i t_j = \frac{1}{i!} (t_{i+1}-1) \ge \frac{1}{i!} (2^{2^{i-1}}-1).
	\end{align*}
	This shows \eqref{vol:lower}. Since $G_i \supseteq \setcond{j e_i}{j = 0,\ldots,t_i}$, we have $\card{G_i \cap \bL} \ge t_i$. The bound \eqref{card:lower} follows after applying \eqref{t:n:bounds}.
\end{proof}

\begin{acknowledgements*}
	I would like to thank A. M. Kasprzyk, Ch. Wagner and the anonymous referees for pointers to the literature and for comments that helped to improve the presentation.
\end{acknowledgements*}

%\bibliographystyle{amsalpha}
%\bibliography{literature}

\begin{thebibliography}{AWW11}

\bibitem[AWW11]{arXiv:1010.1077}
G.~Averkov, Ch. Wagner, and R.~Weismantel, \emph{Maximal lattice-free
  polyhedra: finiteness and an explicit description in dimension three}, Math.
  Oper. Res. \textbf{36} (2011), no.~4, 721--742.

\bibitem[Bar02]{MR1940576}
A.~Barvinok, \emph{A {C}ourse in {C}onvexity}, Graduate Studies in Mathematics,
  vol.~54, American Mathematical Society, Providence, RI, 2002.

\bibitem[BB92]{MR1166957}
A.~A. Borisov and L.~A. Borisov, \emph{Singular toric {F}ano three-folds}, Mat.
  Sb. \textbf{183} (1992), no.~2, 134--141.

\bibitem[Bor00]{arxiv:math/0001109}
A.~Borisov, \emph{Convex lattice polytopes and cones with few lattice points
  inside, from a birational geometry viewpoint}, Preprint ArXiv:math/0001109,
  2000.

\bibitem[Cur22]{MR1520110}
D.~R. Curtiss, \emph{On {K}ellogg's {D}iophantine {P}roblem}, Amer. Math.
  Monthly \textbf{29} (1922), no.~10, 380--387.

\bibitem[DPW10]{DelPiaWeismantel10}
A.~Del~Pia and R.~Weismantel, \emph{On convergence in mixed integer
  programming}, Preprint, to appear in \emph{Math. Programming}, 2010.

\bibitem[Ehr55]{MR0066430}
E.~Ehrhart, \emph{Une g\'en\'eralisation du th\'eor\`eme de {M}inkowski}, C. R.
  Acad. Sci. Paris \textbf{240} (1955), 483--485.

\bibitem[Erd50]{MR0043117}
P.~Erd{\H o}s, \emph{On a {D}iophantine equation}, Mat. Lapok \textbf{1}
  (1950), 192--210.

\bibitem[GL87]{MR893813}
P.~M. Gruber and C.~G. Lekkerkerker, \emph{Geometry of {N}umbers}, second ed.,
  North-Holland Mathematical Library, vol.~37, North-Holland Publishing Co.,
  Amsterdam, 1987.

\bibitem[Gru07]{MR2335496}
P.~M. Gruber, \emph{Convex and {D}iscrete {G}eometry}, Grundlehren der
  Mathematischen Wissenschaften [Fundamental Principles of Mathematical
  Sciences], vol. 336, Springer, Berlin, 2007.

\bibitem[Hen83]{MR688412}
D.~Hensley, \emph{Lattice vertex polytopes with interior lattice points},
  Pacific J. Math. \textbf{105} (1983), no.~1, 183--191.

\bibitem[HNP09]{MR2599087}
Ch. Haase, B.~Nill, and S.~Payne, \emph{Cayley decompositions of lattice
  polytopes and upper bounds for {$h^*$}-polynomials}, J. Reine Angew. Math.
  \textbf{637} (2009), 207--216.

\bibitem[HS09]{MR2478059}
Ch. Haase and J.~Schicho, \emph{Lattice polygons and the number {$2i+7$}},
  Amer. Math. Monthly \textbf{116} (2009), no.~2, 151--165.

\bibitem[Kas09]{MR2549542}
A.~M. Kasprzyk, \emph{Bounds on fake weighted projective space}, Kodai Math. J.
  \textbf{32} (2009), no.~2, 197--208.

\bibitem[Kas10]{MR2760660}
\bysame, \emph{Canonical toric {F}ano threefolds}, Canad. J. Math. \textbf{62}
  (2010), no.~6, 1293--1309.

\bibitem[Law91]{MR1116368}
J.~Lawrence, \emph{Finite unions of closed subgroups of the {$n$}-dimensional
  torus}, Applied geometry and discrete mathematics, DIMACS Ser. Discrete Math.
  Theoret. Comput. Sci., vol.~4, Amer. Math. Soc., Providence, RI, 1991,
  pp.~433--441.

\bibitem[LZ91]{MR1138580}
J.~C. Lagarias and G.~M. Ziegler, \emph{Bounds for lattice polytopes containing
  a fixed number of interior points in a sublattice}, Canad. J. Math.
  \textbf{43} (1991), no.~5, 1022--1035.

\bibitem[NZ11]{arXiv:1101.4292}
B.~Nill and G.~M. Ziegler, \emph{Projecting lattice polytopes without interior
  lattice points}, Math. Oper. Res. \textbf{36} (2011), no.~3, 468--473.

\bibitem[Pik01]{MR1996360}
O.~Pikhurko, \emph{Lattice points in lattice polytopes}, Mathematika
  \textbf{48} (2001), no.~1-2, 15--24 (2003). \MR{2004f:52009}

\bibitem[Rab89]{MR1039134}
S.~Rabinowitz, \emph{A census of convex lattice polygons with at most one
  interior lattice point}, Ars Combin. \textbf{28} (1989), 83--96.

\bibitem[Sch93]{MR1216521}
R.~Schneider, \emph{Convex {B}odies: {T}he {B}runn-{M}inkowski {T}heory},
  Encyclopedia of Mathematics and its Applications, vol.~44, Cambridge
  University Press, Cambridge, 1993.

\bibitem[Sco76]{MR0430960}
P.~R. Scott, \emph{On convex lattice polygons}, Bull. Austral. Math. Soc.
  \textbf{15} (1976), no.~3, 395--399.

\bibitem[ZPW82]{MR651251}
J.~Zaks, M.~A. Perles, and J.~M. Wills, \emph{On lattice polytopes having
  interior lattice points}, Elem. Math. \textbf{37} (1982), no.~2, 44--46.

\end{thebibliography}

\providecommand{\bysame}{\leavevmode\hbox to3em{\hrulefill}\thinspace}
\providecommand{\MR}{\relax\ifhmode\unskip\space\fi MR }
\providecommand{\MRhref}[2]{%
  \href{http://www.ams.org/mathscinet-getitem?mr=#1}{#2}
}
\providecommand{\href}[2]{#2}

\end{document}